\documentclass[11pt,a4paper,leqno]{amsart}
\usepackage{calrsfs}  

\swapnumbers
\numberwithin{equation}{section}

\newcommand{\p}{\partial}

\renewcommand{\d}{\mathrm{d}}

\renewcommand{\epsilon}{\varepsilon}
\newcommand{\la}{\langle}
\newcommand{\ra}{\rangle}
\newcommand{\lla}{\langle\!\!\langle\!\!\langle}
\newcommand{\rra}{\rangle\!\!\rangle\!\!\rangle}
\newcommand{\vvbar}{\,|\!\!|\!\!|\!\!|\,}

\renewcommand{\P}{\mathbb{P}}
\newcommand{\E}{\mathbb{E}}
\newcommand{\Cov}{\mathbb{C}\mathrm{ov}}
\newcommand{\Var}{\mathbb{V}\mathrm{ar}}
\newcommand{\1}{1}
\newcommand{\F}{\mathcal{F}}

\newcommand{\Tr}{\mathrm{Tr}}

\newcommand{\g}{\mathbf{g}}
\newcommand{\G}{\mathbf{G}}
\newcommand{\y}{\mathbf{y}}
\newcommand{\0}{\mathbf{0}}
\newcommand{\SSigma}{\mathbf{\Sigma}}
\newcommand{\mmu}{\boldsymbol{\mu}}
\newcommand{\x}{\mathbf{x}}

\newcommand{\I}{\mathcal{I}}
\renewcommand{\L}{\mathcal{L}}
\newcommand{\LLambda}{\mathrm{\Lambda}}

\newcommand{\PP}{\mathrm{P}}
\newcommand{\KK}{\mathrm{K}}
\newcommand{\II}{\mathrm{I}}
 
\begin{document}


\theoremstyle{definition}
\newtheorem{dfn}[equation]{Definition}
\newtheorem{ass}[equation]{Assumption}
\theoremstyle{plain}
\newtheorem{thm}[equation]{Theorem}
\newtheorem{pro}[equation]{Proposition}
\newtheorem{cor}[equation]{Corollary}
\newtheorem{lmm}[equation]{Lemma}
\theoremstyle{definition}
\newtheorem{rem}[equation]{Remark}
\newtheorem{exa}[equation]{Example}



 
\title{Generalized Gaussian Bridges}
\date{\today}

\author[Sottinen, T.]{Tommi Sottinen}
\address{Tommi Sottinen\\
Department of Mathematics and Statistics\\ 
University of Vaasa
P.O.Box 700\\ 
FIN-65101 Vaasa\\ 
Finland}
\email{tommi.sottinen@uwasa.fi}

\author[Yazigi, A.]{Adil Yazigi}
\address{Adil Yazigi\\
Department of Mathematics and Statistics\\ 
University of Vaasa
P.O.Box 700\\ 
FIN-65101 Vaasa\\ 
Finland}
\email{adil.yazigi@uwasa.fi}

\begin{abstract}
A generalized bridge is the law of a stochastic process that is conditioned on 
N linear functionals of its path.  We consider two types of representations
of such bridges: orthogonal and canonical.  

The orthogonal representation is constructed from the entire path of the 
underlying process. Thus, future knowledge of the path is needed. The orthogonal
representation is provided for any continuous Gaussian process.

In the canonical representation the filtrations and the linear spaces generated by 
the bridge process and the underlying process coincide. Thus, no future 
information of the underlying process is needed.  Also, in the semimartingale case 
the canonical bridge representation is related to the enlargement of filtration 
and semimartingale decompositions.  The canonical representation is provided for 
the so-called prediction-invertible Gaussian processes.  All martingales are
trivially prediction-invertible. A typical non-semimartingale example of a
prediction-invertible Gaussian process is the fractional Brownian motion.

We apply the canonical bridges to insider trading.
\end{abstract}

\thanks{
A. Yazigi was funded by the Finnish Doctoral Programme
in Stochastics and Statistics and by the Finnish Cultural Foundation.}

\maketitle

{\small

}


 
\section{Introduction}

Let $X=(X_t)_{t\in[0,T]}$ be a continuous Gaussian process with positive 
definite covariance function $R$, mean function $m$ of bounded variation, and 
$X_0=m(0)$.  We consider the conditioning, or bridging, of $X$ on $N$ linear 
functionals $\G_T=[G_T^i]_{i=1}^N$ of its paths: 
\begin{equation}\label{eq:awints}
\G_T(X) = \int_0^T \g(t)\, \d X_t
= \left[\int_0^T g_i(t)\, \d X_t\right]_{i=1}^N.
\end{equation} 
We assume, without any loss of generality, that the functions $g_i$ are 
linearly independent. Indeed, if this is not the case then the linearly
dependent, or redundant, components of $\g$ can simply be removed from the 
conditioning (\ref{eq:conditioning}) below without changing it.

The integrals in the conditioning (\ref{eq:awints}) are the so-called abstract
Wiener integrals (see Definition \ref{dfn:AWI} later).  
The abstract Wiener integral $\int_0^T g(t)\, \d X_t$ will be well-defined for
functions or generalized functions $g$ that can be approximated by step functions
in the inner  product $\lla \cdot,\cdot \rra$ defined by the covariance $R$ of
$X$ by bilinearly extending the relation $\lla \1_{[0,t)},\1_{[0,s)}\rra = R(t,s)$.
This means that the integrands $g$ are equivalence classes of Cauchy sequences 
of step functions in the norm $\vvbar\cdot\vvbar$ induced by the inner product
$\lla\cdot,\cdot\rra$. Recall that for the case of Brownian motion we have
$R(t,s) =t\wedge s$. Therefore, for the Brownian motion, the equivalence classes
of step functions is simply the space $L^2([0,T])$.   

Informally, the generalized Gaussian bridge $X^{\g;\y}$ is (the law of) the 
Gaussian process $X$ conditioned on the set
\begin{equation}\label{eq:conditioning}
\left\{\int_0^T \g(t)\,\d X_t=\y\right\} =
\bigcap_{i=1}^N\left\{\int_0^T g_i(t)\, \d X_t = y_i\right\}.
\end{equation}
The rigorous definition is given in Definition \ref{dfn:bridges} later.

For the sake of convenience, we will work on the canonical filtered
probability space $(\Omega,\F, \mathbb{F}, \P)$, where $\Omega=C([0,T])$, $\F$
is the Borel  $\sigma$-algebra on $C([0,T])$ with respect to the supremum norm,
and $\P$ is the Gaussian measure  corresponding to the Gaussian coordinate 
process $X_t(\omega)=\omega(t)$: $\P=\P[X\in \,\cdot\,]$.  The filtration 
$\mathbb{F}=(\F_t)_{t\in[0,T]}$ is the intrinsic filtration of the coordinate 
process $X$ that is  augmented with the null-sets and made right-continuous.

\begin{dfn}\label{dfn:bridges}
The \emph{generalized bridge measure} $\P^{\g;\y}$ is the regular conditional law
$$
\P^{\g;\y} =
\P^{\g;\y}\left[X\in \, \cdot\, \right] = 
\P\left[X\in \ \cdot\ \bigg| \int_0^T \g(t)\, \d X_t=\y\right].
$$
A \emph{representation of the generalized Gaussian bridge} 
is any process $X^{\g;\y}$ satisfying
$$
\P\left[X^{\g;\y}\in\,\cdot\,\right] =
\P^{\g;\y}\left[X\in \, \cdot\, \right]
=
\P\left[ X \in\ \cdot\  \bigg| \int_0^T \g(t)\, \d X_t=\y\right].
$$
\end{dfn}  

Note that the conditioning on the $\P$-null-set (\ref{eq:conditioning}) in 
Definition \ref{dfn:bridges} above is not a problem, since the canonical space 
of continuous processes is a Polish space and all Polish spaces are Borel spaces
and thus admit regular conditional laws, cf. \cite{Kallenberg} Theorem A1.2 and
Theorem 6.3.
Also, note that as a \emph{measure} $\P^{\g;\y}$ the generalized Gaussian  
bridge is unique, but it has several different \emph{representations}
$X^{\g;\y}$. Indeed, for any representation of the bridge one can combine it 
with any $\P$-measure-preserving transformation to get a new representation. 

In this paper we provide two different representations for $X^{\g;\y}$.  The 
first representation, given by Theorem \ref{thm:ortho_decomp}, is called the 
\emph{orthogonal representation}. This representation is a simple consequence 
of orthogonal decompositions of Hilbert spaces associated with Gaussian 
processes and it can be constructed for any continuous Gaussian process
for any conditioning functionals. The second representation, given by Theorem 
\ref{thm:canonical-pred-inv}, is called the \emph{canonical representation}.  
This representation is more interesting but also requires more assumptions.  
The canonical representation is dynamically invertible in the sense that the 
linear spaces  $\L_t(X)$ and $\L_t(X^{\g;\y})$ (see Definition 
\ref{dfn:linear_space} later) generated by the process $X$ and its bridge 
representation $X^{\g;\y}$ coincide for all times $t\in[0,T)$.  
This means that at every time point $t\in[0,T)$ the bridge and the underlying
process can be constructed from each others without knowing the future-time
development of the underlying process or the bridge. 
A typical example of a non-semimartingale Gaussian process for which we can
provide the canonically represented generalized bridge is the fractional Brownian 
motion. 

The canonically represented  
bridge $X^{\g;\y}$ can be interpreted as the original process $X$ with an added 
``information drift'' that bridges the process at the final time $T$.  This 
dynamic drift interpretation should turn out to be useful in applications.  We 
give one such application in connection to insider trading in Section 5.  This 
application is, we must admit, a bit classical.

On earlier work related to bridges, we would like to mention first Alili
\cite{Alili}, Baudoin \cite{Baudoin}, Baudoin and Coutin \cite{BaudoinCoutin} and 
Gasbarra et al. \cite{GasbarraSottinenValkeila}. In \cite{Alili}  generalized 
Brownian bridges were considered. It is our opinion that our article extends 
\cite{Alili} considerably, although we do not consider the ``non-canonical 
representations'' of \cite{Alili}. 
Indeed, Alili \cite{Alili} only considered Brownian motion. Our investigation
extends to a large class of non-semimartingale Gaussian processes.  Also, Alili
\cite{Alili} did not give the canonical representation for bridges, i.e. the
solution to the equation (\ref{eq:can brg 1}) was not given.  We solve the 
equation (\ref{eq:can brg 1}) in (\ref{brg mg}).
The article \cite{Baudoin} is, in a sense, more general than this 
article, since we condition on fixed values $\y$, but in \cite{Baudoin} the 
conditioning is on a probability law. However, in \cite{Baudoin} only 
the Brownian  bridge was considered.  In that sense our approach is more 
general. 
In \cite{BaudoinCoutin} and \cite{GasbarraSottinenValkeila} (simple) bridges
were studied in a similar Gaussian setting as in this article. In this article
we generalize the results of and \cite{BaudoinCoutin} and 
\cite{GasbarraSottinenValkeila} to generalized bridges. 
Second, we would like to mention the articles
\cite{CampiCetinDanilova,ChaumontUribeBravo,GasbarraValkeilaVostrikova,
HoyleHugstonMarcrina} that deal with Markovian and L\'evy bridges.

This paper is organized as follows. In Section 2 we recall some Hilbert spaces 
related to Gaussian processes. In Section 3 we give the orthogonal 
representation for the generalized bridge in the general Gaussian setting.  
Section 4 deals with the canonical bridge representation.  First we give the 
representation for Gaussian martingales.  Then we introduce the so-called 
prediction-invertible processes and develop the canonical bridge representation 
for them.  Then we consider invertible Gaussian Volterra processes, such as the 
fractional Brownian motion, as examples of prediction-invertible processes.  
Finally, in Section 5 we apply the bridges to insider trading.  Indeed, the bridge 
process can be understood from the initial enlargement of filtration point of view.
For more information on the enlargement of filtrations we refer to \cite{Chaleyat-MaurelJeulin} and \cite{JeulinYor}.


\section{Abstract Wiener Integrals and Related Hilbert Spaces}
\label{sect:AWI}

In this section $X=(X_t)_{t\in[0,T]}$ is a continuous (and hence 
separable) Gaussian process with positive definite
covariance $R$, mean zero and $X_0=0$.

Definitions \ref{dfn:linear_space} and \ref{dfn:AWI-space} below give us 
two central separable Hilbert spaces connected to separable Gaussian 
processes.

\begin{dfn}\label{dfn:linear_space}
Let $t\in[0,T]$. The \emph{linear space} $\L_t(X)$ is the Gaussian closed  
linear subspace of $L^2(\Omega,\F,\P)$ generated by the random variables 
$X_s$, $s\le t$, i.e. $\L_t(X) = \overline{\mathrm{span}}\{ X_s; s\le t\}$,
where the closure is taken in $L^2(\Omega,\F,\P)$.
\end{dfn}

The linear space is a Gaussian Hilbert space with the inner product 
$\Cov[\cdot , \cdot]$. Note that since $X$ is continuous, $R$ is also 
continuous, and hence $\L_t(X)$ is separable, and any orthogonal basis 
$(\xi_n)_{n=1}^\infty$ of $\L_t(X)$ is a collection of independent 
standard normal random variables. (Of course, since we chose to work on 
the canonical space, $L^2(\Omega,\F,\P)$ is itself a separable Hilbert 
space.)

\begin{dfn}\label{dfn:AWI-space}
Let $t\in[0,T]$. The \emph{abstract Wiener integrand space} 
$\LLambda_t(X)$ is the completion of the linear span of the indicator 
functions $\1_s := \1_{[0,s)}$, $s\le t$, under the inner product
$\lla \cdot,\cdot\rra$ extended bilinearly from the relation
$$
\lla \1_s,\1_u\rra = R(s,u).
$$
\end{dfn}

The elements of the abstract Wiener integrand space are equivalence 
classes of Cauchy sequences $(f_n)_{n=1}^\infty$ of piecewise constant 
functions. The equivalence of $(f_n)_{n=1}^\infty$ and 
$(g_n)_{n=1}^\infty$ means that 
$$
\vvbar f_n-g_n\vvbar \to 0, \mbox{ as } n\to\infty,
$$
where $\vvbar\cdot\vvbar = \sqrt{\lla \cdot,\cdot\rra}$.

\begin{rem}\label{rem:AWI-space}
\begin{enumerate}
\item
The elements of $\LLambda_t(X)$ cannot in general be identified with 
functions as pointed out e.g. by Pipiras and Taqqu \cite{PipirasTaqqu}
for the case of fractional Brownian motion with Hurst index
$H>1/2$. 
However, if $R$ is of bounded variation one can identity the function
space $|\LLambda_t|(X)\subset \LLambda_t(X)$:
$$
|\LLambda_t|(X) = \left\{ f\in\mathbb{R}^{[0,t]}\,;\,
\int_0^t\!\!\!\int_0^t |f(s)f(u)| \, |R|(\d s,\d u) <\infty\right\}.
$$
\item
While one may want to interpret that
$\LLambda_s(X)\subset \LLambda_t(X)$ for $s\le t$ it may happen that 
$f\in \LLambda_t(X)$, but $f\1_s\not\in\LLambda_s(X)$. Indeed, it may be 
that $\vvbar f\1_s\vvbar > \vvbar f\vvbar$. See Bender and Elliott 
\cite{BenderElliott} for an example in the case of fractional Brownian motion. 
\end{enumerate}
\end{rem}

The space $\LLambda_t(X)$ is isometric to $\L_t(X)$.  Indeed, the relation
\begin{equation}\label{eq:AWI-iso}
\I^X_t[\1_s] := X_s, \quad s\le t,
\end{equation}
can be extended linearly into an isometry from $\LLambda_t(X)$ onto 
$\L_t(X)$.  

\begin{dfn}\label{dfn:AWI}
The isometry $\I^X_t:\LLambda_t(X)\to\L_t(X)$ extended from the
relation (\ref{eq:AWI-iso}) is the 
\emph{abstract Wiener integral}. We denote
$$
\int_0^t f(s)\, \d X_s := \I^X_t[f].
$$
\end{dfn}

Let us end this section by noting that the abstract Wiener integral and the
linear spaces are now connected as
\begin{eqnarray*}
\L_t(X) &=& \left\{ \I_t[f]\,;\, f\in \Lambda_t(X) \right\} 
\end{eqnarray*}
In the special case of the Brownian motion this relation reduces to the 
well-known It\^o isometry
\begin{eqnarray*}
\L_t(W) &=& \left\{ \int_0^t f(s)\, \d W_s \,;\, f\in L^2([0,t]) \right\}. 
\end{eqnarray*}


\section{Orthogonal Generalized Bridge Representation}
\label{sect:orthogonal}

Denote by $\lla \g \rra$ the matrix
$$
{\lla \g \rra}_{ij} := \lla g_i,g_j\rra 
:= 
\Cov\left[\int_0^T g_i(t)\, \d X_t \,,\, \int_0^T g_j(t)\, \d X_t\right]. 
$$

Note that $\lla\g\rra$ does not depend on the mean of $X$ nor on the 
conditioned values $\y$: 
$\lla \g\rra$ depends only on the conditioning functions  
$\g=[g_i]_{i=1}^N$ and the covariance $R$.  Also, since $g_i$'s are linearly 
independent and $R$ is positive definite, the matrix ${\lla \g \rra}$ is 
invertible.

\begin{thm}\label{thm:ortho_decomp}
The generalized Gaussian bridge $X^{\g;\y}$ can be represented as
\begin{equation}\label{eq:ortho_bridge}
X^{\g;\y}_t =  X_t - 
{\lla \1_t,\g\rra}^\top{\lla \g\rra}^{-1}
\left(\int_0^T \g(u)\, \d X_u-\y\right).
\end{equation}
Moreover, $X^{\g;\y}$ is a Gaussian process with
\begin{eqnarray*} 
\E\left[X^{\g;\y}_t\right] &=& 
m(t) - {\lla \1_t,\g\rra}^\top{\lla \g\rra}^{-1}
\left(\int_0^T \g(u)\, \d m(u)-\y\right), \\
\Cov\left[X^{\g;\y}_t,X^{\g;\y}_s\right] &=&
\lla \1_t,\1_s\rra -
{\lla \1_t,\g\rra}^\top{\lla \g\rra}^{-1}
\lla \1_s,\g\rra.
\end{eqnarray*} 
\end{thm}

\begin{proof}
It is well-known (see, e.g., \cite[p. 304]{Shiryaev}) from the theory of 
multivariate Gaussian distributions that conditional distributions are Gaussian 
with
\begin{eqnarray*}
\lefteqn{\E\left[X_t\bigg| \int_0^T \g(u)\d X_u=\y\right]}\\
&=& m(t) +{ \lla \1_t,\g\rra}^\top{\lla \g\rra}^{-1}
\left(\y-\int_0^T \g(u)\, \d m(u)\right), \\
\lefteqn{\Cov\left[X_t,X_s\bigg| \int_0^T \g(u)\, \d X_u=\y\right] }\\
&=&
\lla \1_t,\1_s\rra - {\lla \1_t,\g\rra}^\top{\lla \g\rra}^{-1}
\lla \1_s,\g\rra.
\end{eqnarray*}
The claim follows from this.
\end{proof}

\begin{cor}\label{cor:m-x}
Let $X$ be a centered Gaussian process with $X_0=0$ and let $m$ be a function
of bounded variation. Denote $X^{\g}:=X^{\g;\0}$, i.e., $X^{\g}$ is conditioned
on $\{\int_0^T\g(t)\d X_t =\0\}$. Then
\begin{eqnarray*}
(X+m)^{\g;\y}_t &=& X^{\g}_t 
\, \,  + \left(m(t)-{\lla \1_t,\g\rra}^\top{\lla\g\rra}^{-1}
\int_0^T \g(u)\, \d m(u)\right) \\
& & 
+ \, \, {\lla \1_t,\g\rra}^\top{\lla \g\rra}^{-1}\y.
\end{eqnarray*}
\end{cor}

\begin{rem}
Corollary \ref{cor:m-x} tells us how to construct, by adding a
deterministic drift, a general bridge from a bridge that is constructed 
from a centered process with conditioning $\y=\0$.
So, in what follows, we shall almost always assume that the process
$X$ is centered, i.e. $m(t)=0$, and all conditionings are to $\y=\0$.
\end{rem}

\begin{exa}\label{exa:Bb_sup} 
Let $X$ be a zero mean Gaussian process with covariance $R$.
Consider the conditioning on the final value and the average value:
\begin{eqnarray*}
X_T &=& 0, \\
\frac{1}{T}\int_0^T X_t \,\d t &=& 0.
\end{eqnarray*}  
This is a generalized Gaussian bridge.  Indeed, 
\begin{eqnarray*}
X_T &=& \int_0^T  1\,\d X_t 
=: \int_0^T g_1(t) \,\d X_t, \\
\frac{1}{T}\int_0^T X_t\,\d t &=&
\int_0^T \frac{T-t}{T}\,\d X_t 
=:
\int_0^T g_2(t) \,\d X_t. 
\end{eqnarray*}
Now,
\begin{eqnarray*}
\lla \1_t,g_1\rra &=& \E\left[ X_tX_T\right] 
= R(t,T), \\
\lla \1_t,g_2\rra &=& 
\E\left[X_t\, \frac{1}{T}\int_0^T X_s\,\d s\right] 
= \frac{1}{T}\int_0^T R(t,s)\, \d s, 
\end{eqnarray*}
\begin{eqnarray*}
\lla g_1,g_1 \rra &=& \E\left[X_TX_T\right] 
= R(T,T), \\
\lla g_1,g_2\rra &=& 
\E\left[X_T\, \frac{1}{T}\int_0^T X_s\, \d s\right] 
= \frac{1}{T}\int_0^T R(T,s)\, \d s, \\
\lla g_2,g_2\rra &=& 
\E\left[\frac{1}{T}\int_0^T X_s\,\d s\, 
\frac{1}{T}\int_0^T X_u\, \d u\right] 
=\frac{1}{T^2}\int_0^T\!\!\!\int_0^T R(s,u)\, \d u\d s, 
\end{eqnarray*}
$$
|\lla \g\rra| = \frac{1}{T^2} 
\int_0^T\int_0^T   R(T,T)R(s,u)-R(T,s)R(T,u) \,  \d u\,\d s
$$
and
$$ {\lla \g \rra}^{-1}=\frac{1}{|\lla \g\rra|}
\left[\begin{array}{rr}
\lla g_2,g_2\rra & -\lla g_1,g_2\rra \\
- \lla g_1,g_2\rra  & \lla g_1,g_1 \rra
\end{array}\right].
$$
Thus, by Theorem \ref{thm:ortho_decomp}, 
\begin{eqnarray*}
X_t^\g
&=& 
X_t-\frac{\lla \1_t,g_1\rra \lla g_2,g_2\rra - 
\lla \1_t,g_2\rra \lla g_1,g_2\rra}{|\lla \g\rra|}\int_0^T g_1(t) \,\d X_t \\
& & -\frac{\lla \1_t,g_2\rra \lla g_1,g_1\rra - 
\lla \1_t,g_1\rra \lla g_1,g_2\rra}{|\lla \g\rra|}\int_0^T g_2(t) \,\d X_t \\
&=& X_t - 
\frac{\int_0^T\int_0^T R(t,T)R(s,u)-R(t,s)R(T,s)\d s\, \d u }{\int_0^T  \int_0^T 
R(T,T)R(s,u)-R(T,s)R(T,u)  \, \d s \, \d u }X_T\\
& &- \frac{ T \int_0^T R(T,T)R(t,s)-R(t,T)R(T,s) \d s}{\int_0^T  \int_0^T 
R(T,T)R(s,u)-R(T,s)R(T,u)  \, \d s \, \d u }\int_0^T \frac{T-t}{T}\,\d X_t.
\end{eqnarray*}
\end{exa}

\begin{rem}
\begin{enumerate}
\item
Since Gaussian conditionings are projections in Hilbert space to a subspace, 
it is well-known that they can be done iteratively. Indeed, let 
$X^n:=X^{g_1,\ldots,g_n;y_1,\ldots,y_n}$ and let $X^0 := X$ be the original 
process.  Then the orthogonal generalized  bridge representation $X^N$ can be 
constructed from the rule
$$
X_t^n = X_t^{n-1} - 
\frac{{\lla \1_t,g_n\rra}_{n-1}}{{\lla g_n,g_n\rra}_{n-1}}
\left[\int_0^T g_n(u)\, \d X_u^{n-1}-y_n\right],
$$
where $\lla\cdot,\cdot\rra_{n-1}$ is the inner product in $\L_T(X^{n-1})$.
\item
If the conditioning variables $g_j$ are indicator functions $\1_{t_j}$ then the 
corresponding generalized bridge is a \emph{multibridge}. That is, it is pinned 
down to values $y_j$ at points $t_j$. For the multibridge 
$X^N=X^{\1_{t_1},\ldots,\1_{t_N};y_1,\ldots,y_N}$ 
the orthogonal bridge decomposition can be constructed from the iteration
\begin{eqnarray*}
X^0_t &=& X_t, \\
X^n_t &=& X^{n-1}_t -
\frac{R_{n-1}(t,t_n)}{R_{n-1}(t_n,t_n)}
\left[X^{n-1}_{t_n}-y_{n}\right],
\end{eqnarray*}
where
\begin{eqnarray*}
R_0(t,s) &=& R(t,s), \\
R_n(t,s) &=& R_{n-1}(t,s) - 
\frac{R_{n-1}(t,t_n)R_{n-1}(t_n,s)}{R_{n-1}(t_n,t_n)}.
\end{eqnarray*}
\end{enumerate}
\end{rem}

\section{Canonical Generalized Bridge Representation}
\label{sect:canonical}

The problem with the orthogonal bridge representation (\ref{eq:ortho_bridge})
of $X^{\g;\y}$ is that in order to construct it at any point $t\in[0,T)$
one needs the whole path of the underlying process $X$ up to time $T$.
In this section we construct a bridge representation that is canonical in the 
following sense:

\begin{dfn}\label{dfn:canonical}
The bridge $X^{\g;\y}$ is of \emph{canonical representation} if, for all 
$t \in [0,T)$, $X_t^{\g;\y} \in \L_t(X)$ and $X_t \in \L_t(X^{\g;\y})$.
\end{dfn}

\begin{exa}\label{exa:bb}
Consider the classical Brownian bridge. That is, condition the Brownian motion
$W$ with  $\g=g=1$. Now, the orthogonal representation is 
$$
W_t^1 = W_t - \frac{t}{T}W_T.
$$
This is not a canonical representation, since the future knowledge $W_T$ is 
needed to construct $W_t^1$ for any $t\in(0,T)$. A canonical representation for the 
Brownian bridge is, by calculating the $\ell^*_\g$ in 
Theorem \ref{thm:M-bridge-canonical} below, 
\begin{eqnarray*}
W_t^1 &=& W_t - \int_0^t\int_0^s \frac{1}{T-u} \, \d W_u \, \d s \\ 
&=& (T-t)\int_0^t \frac{1}{T-s}\, \d W_s.
\end{eqnarray*}
\end{exa}

\begin{rem}
Since the conditional laws of Gaussian processes are Gaussian and Gaussian 
spaces are linear, the assumptions $X_t^{\g;\y} \in \L_t(X)$ and 
$X_t \in \L_t(X^{\g;\y})$ of Definition \ref{dfn:canonical} are the same as 
assuming that $X^{\g;\y}_t$ is $\F_t^{X}$-measurable and $X_t$ is 
$\F^{X^{\g;\y}}_t$-measurable (and, consequently, 
$\F_t^{X}=\F_t^{X^{\g;\y}}$).  This fact is very special to Gaussian
processes.  Indeed, in general conditioned processes such as generalized
bridges are not linear transformations of the underlying process.
\end{rem}

We shall require that the restricted measures 
$\P^{\g,\y}_t:= \P^{\g;\y}|\F_t$ and $\P_t:=\P|\F_t$ are equivalent for all
$t<T$ (they are obviously singular for $t=T$). To this end we assume that
the matrix
\begin{eqnarray}\label{eq:g}
{\lla\g\rra}_{ij}(t) &:=&
\E\left[\Big(G_T^i(X)-G_t^i(X)\Big)\Big(G_T^j(X)-G_t^j(X)\Big)\right] \\
\nonumber
&=& 
\E\left[\int_t^Tg_i(s)\, \d X_s \int_t ^T g_j(s)\, \d X_s\right]
\end{eqnarray}
is invertible for all $t<T$.

\begin{rem}
On notation: in the previous section we considered the matrix $\lla\g\rra$,
but from now on we consider the function $\lla\g\rra(\cdot)$.  Their
connection is of course $\lla \g \rra = \lla\g\rra(0)$. We hope that
this overloading of notation does not cause confusion to the reader. 
\end{rem}

\subsection*{Gaussian Martingales}
We first construct the canonical representation when the underlying process 
is a continuous Gaussian martingale $M$ with strictly increasing bracket 
$\la M\ra$ and $M_0=0$.  Note that the bracket is strictly increasing if and 
only if the covariance $R$ is positive definite. Indeed, for Gaussian 
martingales we have $R(t,s) = \Var(M_{t\wedge s})=\la M\ra_{t\wedge s}$.   

Define a Volterra kernel
\begin{equation}\label{eq:lg} 
\ell_\g(t,s) := -\g^\top(t)\,{\lla \g \rra}^{-1}(t)\,\g(s).
\end{equation}
Note that the kernel $\ell_\g$ depends on the process $M$ through its 
covariance $\lla \cdot,\cdot \rra$, and in the Gaussian martingale case
we have
$$
{\lla \g \rra}_{ij}(t) =
\int_t^T g_i(s)g_j(s) \, \d\la M \ra_s.
$$

The following Lemma \ref{lmm:canon_radon_nikodym} is the key observation in 
finding the canonical generalized bridge representation.  Actually, it is a 
multivariate version of Proposition 6 of \cite{GasbarraSottinenValkeila}.

\begin{lmm}\label{lmm:canon_radon_nikodym}
Let $\ell_\g$ be given by (\ref{eq:lg}) and let $M$ be a continuous Gaussian
martingale with strictly increasing bracket $\la M\ra$ and $M_0=0$. Then
the Radon-Nikodym derivative $\d\P^\g_t/\d\P_t$ can be expressed in the 
form
\begin{equation*}
\frac{\d\P^\g_t}{\d\P_t} = 
\exp\left\{\int_0^t \int_0^s \ell_\g(s,u)\, \d M_u \d M_s\
-\frac12\int_0^t \left( \int_0^s \ell_\g(s,u)\, \d M_u\right)^2 
\d\la M\ra_s\right\}
\end{equation*}
for all $t \in [0,T)$.
\end{lmm}

\begin{proof}
Let 
$$
p(\y;\mmu,\SSigma) := \frac{1}{(2\pi)^{N/2}|\SSigma|^{1/2}}
\exp\left\{-\frac12(\y-\mmu)^\top\SSigma^{-1}(\y-\mmu)\right\}
$$
be the Gaussian density on $\mathbb{R}^N$ and let
$$
\alpha_t^\g(\d\y) := \P\left[\G_T(M) \in \d\y \ \Big| \ \F_t^M\right]
$$
be the conditional law of the conditioning functionals
$\G_T(M) = \int_0^T \g(s)\, \d M_s$ given the information $\F_t^M$.
 
First, by the Bayes' formula, we have
$$
\frac{\d\P^\g_t}{\d\P_t} = \frac{\d\alpha_t^\g}{\d\alpha_0^\g}(\0).
$$

Second, by the martingale property, we have
$$
\frac{\d\alpha_t^\g}{\d\alpha_0^\g}(\0)
= \frac{p\Big(\0;\G_t(M),\lla \g \rra(t)\Big)}{
p\Big(\0;\G_0(M),\lla \g \rra(0)\Big)},
$$
where we have denoted $\G_t(M) = \int_0^t \g(s)\,\d M_s$.

Third, denote
$$
\frac{p\Big(\0;\G_t(M),\lla \g \rra(t)\Big)}{
p\Big(\0;\G_0(M),\lla \g \rra(0)\Big)}
=: \left( \frac{|\lla \g \rra|(0)}{|\lla \g \rra|(t)}\right)^\frac12
\exp\left\{F(t,\G_t(M))-F(0,\G_0(M)\right\},
$$
with
$$
F(t,\G_t(M))=-\frac12\left(\int_0^t \g(s)\, \d M_s\right)^\top\lla \g \rra^{-1}(0)
\left(\int_0^t \g(s) \, \d M_s \right).
$$
Then, straightforward differentiation yields
\begin{eqnarray*}
\int_0^t\frac{\p F}{\p s}(s,\G_s(M)) \, \d s
&=& -\frac12\int_0^t\left( \int_0^s \ell_\g(s,u)\, \d M_u\right)^2 \, 
\d\la M\ra_s,\\
\int_0^t \frac{\p F}{\p x}(s,\G_s(M)) \, \d M_s
&=& \int_0^t  \int_0^s \ell_\g(s,u) \, \d M_u\, \d M_s,\\
-\frac12 \int_0^t \frac{\p^2 F}{\p x^2}(s,\G_s(M)) \, \d\la M\ra_s
 &=& \log\left(\frac{|\lla \g \rra|(t)}{|\lla\g \rra|(0)}\right)^\frac12
\end{eqnarray*}
and the form of the Radon--Nikodym derivative follows 
by applying the It\^o formula. 
\end{proof}
    
\begin{cor}\label{cor:Mg}  
The canonical bridge representation $M^\g$ satisfies the stochastic 
differential equation
\begin{equation}\label{eq:can brg 1}
\d M_t= \d M^\g_t-\int_0^t \ell_\g(t,s)\, \d M^\g_s\, \d \la M\ra_t,
\end{equation}
where $\ell_\g$ is given by (\ref{eq:lg}). Moreover  $\la M\ra = \la M^\g\ra$.

\end{cor}
\begin{proof}
The claim follows by using the Girsanov's theorem. 
\end{proof}

\begin{rem}
\begin{enumerate}
\item
Note that for all $\epsilon>0$,
$$
\int_0^{T-\epsilon} \int_0^t \ell_\g(t,s)^2 \, \d\la M\ra_s \, \d\la M\ra_t < \infty. 
$$
In view of (\ref{eq:can brg 1}) this means that the processes $M$ and $M^\g$
are equivalent in law on $[0,T-\epsilon]$ for all $\epsilon>0$.  Indeed, equation (\ref{eq:can brg 1})
can be viewed as the \emph{Hitsuda representation} between two equivalent 
Gaussian processes, cf. Hida and Hitsuda \cite{HidaHitsuda}. Also note that
$$
\int_0^{T} \int_0^t \ell_\g(t,s)^2 \, \d\la M\ra_s \, \d\la M\ra_t = \infty
$$
meaning that the measures $\P$ and $\P^\g$ are singular on $[0,T]$.
\item
In the case of the Brownian bridge, cf. Example \ref{exa:bb}, the item (i) above
can be clearly seen. Indeed, 
$$
\ell_\g(t,s) = \frac{1}{T-t}
$$
and $\d\la W\ra_s = \d s$. 
\item
In the case of $\y\ne\0$, the formula (\ref{eq:can brg 1}) takes the form
\begin{equation}\label{eq:can brg 2}
\d M_t= \d M^{\g;\y}_t+
\left(\g^\top(t){\lla \g \rra}^{-1}(t)\y-
\int_0^t \ell_\g(t,s)\, \d M^{\g;\y}_s\right)\, \d \la M\ra_t.
\end{equation}
\end{enumerate}
\end{rem}

Next we solve the stochastic differential equation (\ref{eq:can brg 1}) of 
Corollary \ref{cor:Mg}. In general, solving a Volterra--Stieltjes equation like 
(\ref{eq:can brg 1}) in a closed form is difficult.  Of course, the general 
theory of Volterra equations suggests that the solution will be of the form 
(\ref{brg mg}) of Theorem \ref{thm:M-bridge-canonical} below, where $\ell_\g^*$
is the resolvent kernel of $\ell_\g$ determined by the resolvent equation
(\ref{eq:resolvent}) given below.  Also, the general theory suggests that the 
resolvent kernel can be calculated implicitly by using the Neumann series. In 
our case the kernel $\ell_\g$ is a  quadratic form that factorizes in its 
argument. This allows us to calculate the resolvent $\ell_\g^*$ explicitly as 
(\ref{eq:lg-resolvent}) below. 
(We would like to point out that a similar SDE was treated in \cite{AliliWu}
and \cite{HibinoHitsudaMuraoka}.)

\begin{thm}\label{thm:M-bridge-canonical}
Let $s\le t \in[0,T]$. Define the Volterra kernel
\begin{eqnarray}\label{eq:lg-resolvent}
\ell_\g^*(t,s) &:=&
-\ell_\g(t,s)\frac{|\lla \g\rra|(t)}{|\lla \g\rra|(s)} \\ \nonumber
&=&
|{\lla \g \rra}|(t)\g^\top(t){\lla \g \rra}^{-1}(t) \,\,
\frac{\g(s)}{|{\lla \g \rra}|(s)}.
\end{eqnarray}
Then the bridge $M^\g$ has the canonical representation
\begin{equation}\label{brg mg}
\d M^\g_t=\d M_t-\int_0^t \ell_\g^*(t,s) \, \d M_s\, \d \la M\ra_t,
\end{equation}
i.e., (\ref{brg mg}) is the solution to (\ref{eq:can brg 1}).
\end{thm}

\begin{proof}
Equation (\ref{brg mg}) is the solution to (\ref{eq:can brg 1})
if the kernel $\ell_\g^*$ satisfies the 
\emph{resolvent equation} 
\begin{equation}\label{eq:resolvent}
\ell_\g(t,s)+\ell_\g^*(t,s) =
\int_s^t \ell_\g(t,u)\ell_\g^*(u,s) \, \d \la M\ra_u.
\end{equation}
This well-known if $\d\la M\ra_u=\d u$, cf. e.g. Riesz and Sz.-Nagy
\cite{RieszSzNagy}. In the $\d\la M\ra$ case the resolvent equation 
can be derived as in the classical $\d u$ case. We show the derivation here,
for the convenience of the reader:

Suppose (\ref{brg mg}) is the solution to (\ref{eq:can brg 1}).
This means that 
\begin{eqnarray*}
\d M_t &=& \left(\d M_t - 
\int_0^t \ell_\g^*(t,s)\,\d M_s\,\d\la M\ra_t\right) \\
& &- \int_0^t \ell_\g(t,s)
\left(\d M_s - 
\int_0^s \ell_\g^*(s,u)\,\d M_u\,\d\la M\ra_s\right)\d \la M\ra_t,
\end{eqnarray*}
or, in the integral form, by using the Fubini's theorem,
\begin{eqnarray*}
M_t &=& M_t - 
\int_0^t\!\int_s^t \ell_\g^*(u,s)\,\, \d\la M \ra_u \d M_s  \\
& &
-\int_0^t\!\int_s^t\ell_\g(u,s)\,\,\d\la M\ra_u\d M_s \\
& & + \int_0^t\!\int_s^t\,\int_u^s \ell_\g(s,v)\ell_\g^*(v,u)\d\la M\ra_v 
\,\,\d\la M\ra_u\d M_s.
\end{eqnarray*}
The resolvent criterion (\ref{eq:resolvent}) follows by identifying the 
integrands in the $\d\la M\ra_u\d M_s$-integrals above.

Finally, let us check that the resolvent equation (\ref{eq:resolvent}) is 
satisfied with $\ell_\g$ and $\ell_\g^*$ defined by (\ref{eq:lg}) and 
(\ref{eq:lg-resolvent}), respectively: 
\begin{eqnarray*}
\lefteqn{\int_s^t \ell_\g(t,u)\ell_\g^*(u,s)\,\d\la M \ra_u} \\
&=& -\int_s^t \g^\top(t) {\lla \g\rra}^{-1}(t)\g(u) \,\,
|\lla \g \rra|(u)\g^\top(u) {\lla \g\rra}^{-1}(u)\frac{\g(s)}{
|\lla \g \rra|(s)}\,\d\la M \ra_u \\
&=&
-\g^\top(t) {\lla\g\rra}^{-1}(t) \frac{\g(s)}{|\lla\g \rra|(s)}
\int_s^t \g(u)
|\lla \g \rra|(u)\g^\top(u) {\lla \g\rra}^{-1}(u)
\,\d\la M \ra_u \\
&=&
\g^\top(t) {\lla\g\rra}^{-1}(t) \frac{\g(s)}{|\lla \g \rra|(s)}
\int_s^t {\lla \g \rra}^{-1}(u)|\lla \g \rra|(u)\d\lla \g \rra(u) \\
&=&
\g^\top(t) {\lla\g\rra}^{-1}(t) \frac{\g(s)}{|\lla \g \rra|(s)}
\Big(|\lla \g \rra|(t)-|\lla \g \rra|(s)\Big) \\
&=& 
\g^\top(t) {\lla \g \rra}^{-1}(t) \g(s)
\frac{|\lla \g \rra|(t)}{|\lla \g \rra|(s)}
-\g^\top(t) {\lla \g \rra}^{-1}(t) \g(s) \\
&=&
\ell_\g^*(t,s)+\ell_\g(t,s),
\end{eqnarray*}
since
$$
\d{\lla \g \rra}(t) = -\g^\top(t)\g(t)\d\la M \ra_t. 
$$
So, the resolvent equation (\ref{eq:resolvent}) holds.
\end{proof}

\subsection*{Gaussian Prediction-Invertible Processes} 

To construct a canonical representation for bridges of Gaussian 
non-semimartingales is problematic, since we cannot apply stochastic calculus to
non-semimartingales. In order to invoke the stochastic calculus we need to 
associate the Gaussian non-semimartingale with some martingale.  A natural 
martingale associated with a stochastic process is its prediction martingale:  

For a (Gaussian) process $X$ its \emph{prediction martingale} is
the process $\hat X$ defined as
$$
\hat X_t = \E\left[X_T\big|\F_t^X\right].
$$ 
Since for Gaussian processes $\hat X_t \in \L_t(X)$, we may write, at
least informally, that
$$
\hat X_t = \int_0^t p(t,s)\, \d X_s,
$$
where the abstract kernel $p$ depends also on $T$
(since $\hat X$ depends on $T$). In Definition \ref{dfn:pred-inv}
below we assume that the kernel $p$ exists as a real, and not only formal, 
function.  We also assume that the kernel $p$ is invertible.

\begin{dfn}\label{dfn:pred-inv}
A Gaussian process $X$ is \emph{prediction-invertible} if there exists a kernel 
$p$ such that its prediction martingale $\hat X$ is continuous, can be 
represented as
$$
\hat X_t = \int_0^t p(t,s)\, \d X_s, 
$$
and there exists an inverse kernel $p^{-1}$ such that, for all
$t\in[0,T]$, $p^{-1}(t,\cdot)\in L^2([0,T],\d\langle\hat X\rangle)$ and 
$X$ can be recovered from $\hat X$ by
$$
X_t = \int_0^t p^{-1}(t,s)\, \d \hat X_s.
$$ 
\end{dfn}

\begin{rem}
In general it seems to be a difficult problem to determine whether a Gaussian
process is prediction-invertible or not.  In the discrete time non-degenerate
case all Gaussian processes are prediction-invertible. In continuous time the 
situation is more difficult, as Example \ref{exa:non-pred-inv} below 
illustrates. Nevertheless, we can immediately see that if the centered Gaussian 
process $X$ with covariance $R$ is prediction-invertible, then the covariance 
must satisfy the relation
$$
R(t,s) = \int_0^{t\wedge s} p^{-1}(t,u)\, p^{-1}(s,u)\, \d\la \hat X\ra_u,
$$
where the bracket $\la \hat X \ra$ can be calculated as the variance of the
conditional expectation:
$$
\la \hat X \ra_u = \Var\left(\E\left[ X_T|\F_u\right]\right).
$$
However, this criterion does not seem to be very helpful in practice.
\end{rem}

\begin{exa}\label{exa:non-pred-inv}
Consider the Gaussian slope $X_t = t\xi$, $t\in[0,T]$, where $\xi$ is a 
standard normal random variable.  Now, if we consider the ``raw filtration''
$\mathcal{G}_t^X = \sigma(X_s; s\le t)$, then $X$ is not prediction 
invertible.  Indeed, then $\hat X_0 = 0$ but $\hat X_t=X_T$, if $t\in(0,T]$.  
So, $\hat X$ is not continuous. On the other hand, the augmented filtration is 
simply $\F_t^X=\sigma(\xi)$ for all $t\in[0,T]$. So, $\hat X = X_T$.  Note, 
however, that in both cases the slope $X$ can be recovered from the prediction 
martingale: $X_t = \frac{t}{T} \hat X_t$.
\end{exa}

In order to represent abstract Wiener integrals of $X$ in terms of
Wiener--It\^o integrals of $\hat X$ we need to extend the kernels $p$ and
$p^{-1}$ to linear operators:

\begin{dfn} \label{def:p operator}
Let $X$ be prediction-invertible. Define operators $\PP$ and $\PP^{-1}$
by extending linearly the relations
\begin{eqnarray*}
\PP[\1_t] &=& p(t,\cdot), \\
\PP^{-1}[\1_t] &=& p^{-1}(t,\cdot).
\end{eqnarray*}
\end{dfn}

Now the following lemma is obvious.

\begin{lmm}\label{lmm:P-isometry}
Let $f$ be such a function that 
$\PP^{-1}[f] \in L^2([0,T],\d\la\hat X\ra)$ and let
$\hat g\in L^2([0,T],\d\la \hat X\ra)$. Then
\begin{eqnarray}\label{eq:Pminus}
\int_0^T f(t)\, \d X_t &=& \int_0^T \PP^{-1}[f](t)\, \d\hat X_t, \\
\label{eq:P}
\int_0^T \hat g(t)\, \d\hat X_t &=& \int_0^T \PP[\hat g](t)\, \d X_t.
\end{eqnarray}
\end{lmm}

\begin{rem}
\begin{enumerate}
\item 
Equation (\ref{eq:Pminus}) or (\ref{eq:P}) can actually be taken as the 
definition of the Wiener integral with respect to $X$.
\item
The operators $\PP$ and $\PP^{-1}$ depend on $T$.  
\item
If $p^{-1}(\cdot,s)$ has bounded variation, we can represent $\PP^{-1}$ as
$$
\PP^{-1}[f](s) = f(s)p^{-1}(T,s) +
\int_s^T \left( f(t) - f(s)\right)\, p^{-1}(\d t,s).
$$
Similar formula holds for $\PP$ also, if $p(\cdot,s)$ has bounded variation.
\item Let ${\lla\g\rra}^X(t)$ denote the remaining covariance matrix with
respect to $X$, i.e.,
$$
{\lla\g\rra}^{X}_{ij}(t) =
\E\left[\int_t^T g_i(s)\,\d X_s \, \int_t^T g_j(s)\, \d X_s\right]. 
$$
Let ${\lla \hat \g \rra}^{\hat X}(t)$ denote the remaining covariance matrix
with respect to $\hat X$, i.e.,
$$
{\lla \hat \g \rra}^{\hat X}_{ij}(t) =
\int_t^T \hat g_i(s) \hat g_j(s)\, \d\la \hat X\ra_s. 
$$
Then
$$
{\lla \g \rra}^{X}_{ij}(t) 
=
{\lla \PP^{-1}[\g] \rra}^{\hat X}_{ij}(t) 
=
\int_t^T \PP^{-1}[g_i](s)\PP^{-1}[g_j](s)\, \d\la \hat X\ra_s. 
$$
\end{enumerate}
\end{rem}

Now, let $X^\g$ be the bridge conditioned on $\int_0^T \g(s)\, \d X_s=\0$. 
By Lemma \ref{lmm:P-isometry} we can rewrite the conditioning as
\begin{equation}\label{eq:new_cond}
\int_0^T \g(t) \, \d X_t=\int_0^T \PP^{-1}[\g](t) \, \d \hat{X}(t) = \0,
\end{equation}
With this observation the following theorem, that is the main result of this 
article, follows.

\begin{thm}\label{thm:canonical-pred-inv}
Let $X$ be prediction-invertible Gaussian process.  Assume that, for all
$t\in [0,T]$ and $i=1,\ldots,N$,  $g_i\1_t\in \Lambda_t(X)$. Then the 
generalized bridge $X^\g$ admits the canonical representation
\begin{equation}\label{brg pre inv}
X^\g_t = X_t -
\int_0^t\int_s^t p^{-1}(t,u)\PP\left[\hat\ell_{\hat\g}^*(u,\cdot)\right](s) 
\, \d \la \hat{X}\ra_u \, \d  X_s,
\end{equation}
where
\begin{eqnarray*}
\hat g_i &=& \PP^{-1}[g_i], \\
\hat\ell_{\hat\g}^*(u,v) &=&
|{\lla \hat\g \rra}^{\hat X}|(u)\hat\g^\top(u)
{(\lla \hat\g \rra^{\hat X})}^{-1}(u)\,
\frac{\hat\g(v)}{|\lla \hat\g \rra^{\hat X}|(v)}, \\
\lla \hat\g \rra^{\hat X}_{ij}(t) &=&
\int_t^T \hat g_i(s)\hat g_j(s)\, \d\la \hat X\ra_s 
= \lla \g \rra^{X}_{ij}(t).
\end{eqnarray*}
\end{thm}

\begin{proof}
Since $\hat X$ is a Gaussian martingale and because of the equality 
(\ref{eq:new_cond})
we can use Theorem  \ref{thm:M-bridge-canonical}.  We obtain
$$
\d \hat{X}^{\hat{\g}}_s =
\d \hat{X}_s-\int_0^s \hat\ell_{\hat{\g}}^*(s,u) \, 
\d \hat{X}_u \, \d \la \hat{X}\ra_s.
$$

Now, by using the fact that $X$ is prediction invertible, we can recover $X$
from $\hat{X}$, and consequently also $X^\g$ from $\hat{X}^{\hat\g}$ by
operating with the kernel $p^{-1}$ in the following way:
\begin{eqnarray}
X^\g_t 
&=& 
\int_0^t p^{-1}(t,s)\,  \d \hat{X}^{\hat{\g}}_s \nonumber \\
&=& 
X_t-\int_0^t  p^{-1}(t,s)\left(\int_0^s 
\hat\ell_{\hat{\g}}^*(s,u)\, \d \hat{X}_u\right)
\, \d \la \hat{X}\ra_s. \label{eq:sol}
\end{eqnarray}
In a sense the representation (\ref{eq:sol}) is a canonical representation
already. However, let us write it in terms of $X$  instead of $\hat{X}$. We 
obtain
\begin{eqnarray*}
X^\g_t 
&=&
X_t-\int_0^t   p^{-1}(t,s) \int_0^s 
\PP\left[\hat\ell_{\hat\g}^*(s,\cdot)\right](u) \,\d X_u\,\d\la\hat{X}\ra_s \\
&=&  
X_t-\int_0^t\int_s^t p^{-1}(t,u) 
\PP\left[\hat\ell_{\hat\g}^*(u,\cdot)\right](s) 
\, \d \la \hat{X}\ra_u\, \d  X_s.
\end{eqnarray*}
\end{proof}

\begin{rem}
Recall that, by assumption, the processes $X^\g$ and $X$ are equivalent on
$\F_t$, $t<T$. So, the representation (\ref{brg pre inv}) is an analogue of
the Hitsuda representation for prediction-invertible processes.  Indeed, one
can show, just like in \cite{Sottinen,SottinenTudor}, that a zero mean
Gaussian process $\tilde X$ is equivalent in law to the zero mean 
prediction-invertible Gaussian process $X$ if it admits the representation 
$$
\tilde X_t = X_t - \int_0^t f(t,s) \,\d X_s
$$
where 
$$
f(t,s) = \int_s^t p^{-1}(t,u)\PP\left[v(u,\cdot)\right](s)
\,\d \la \hat X\ra_u
$$
for some Volterra kernel 
$v\in L^2([0,T]^2, \d\la \hat X\ra \otimes \d\la \hat X\ra)$.
\end{rem}


It seems that, except in \cite{GasbarraSottinenValkeila}, the prediction-invertible
Gaussian processes have not been studied at all.  Therefore, we give a class 
of prediction-invertible processes that is related to a class that has been
studied in the literature: the Gaussian Volterra processes.  See, e.g.,
Al\`os et al. \cite{AlosMazetNualart}, for a study of stochastic calculus with
respect to Gaussian Volterra processes.

\begin{dfn}\label{dfn:gvp}
$V$ is an \emph{invertible Gaussian Volterra process}
if it is continuous and there exist Volterra kernels $k$ and $k^{-1}$ such that
\begin{eqnarray}\label{eq:gvp1}
V_t &=& \int_0^t k(t,s)\, \d W_s, \\
W_t &=& \int_0^t k^{-1}(t,s)\, \d V_s. \label{eq:gvp2}
\end{eqnarray} 
Here $W$ is the standard Brownian motion, 
$k(t,\cdot)\in L^2([0,t])=\Lambda_t(W)$ and
$k^{-1}(t,\cdot)\in \Lambda_t(V)$ for all $t\in [0,T]$.
\end{dfn}

\begin{rem}
\begin{enumerate}
\item
The representation (\ref{eq:gvp1}), defining a Gaussian Volterra process,
states that the covariance $R$ of $V$ can be written as
$$
R(t,s) = \int_0^{t\wedge s} k(t,u)k(s,u)\, \d u.
$$
So, in some sense, the kernel $k$ is the square root, or the Cholesky
decomposition, of the covariance $R$.  
\item
The inverse relation (\ref{eq:gvp2}) means that the indicators
$\1_t$, $t\in [0,T]$, can be approximated in $L^2([0,t])$
with linear combinations of the functions $k(t_j,\cdot)$, 
$t_j\in [0,t]$. I.e., the indicators $\1_t$ belong to the image
of the operator $\KK$ extending the kernel $k$ linearly as
discussed below. 
\end{enumerate}
\end{rem}

Precisely as with the kernels $p$ and $p^{-1}$, we can define the operators 
$\KK$ and $\KK^{-1}$ by linearly extending the relations
$$
\KK[\1_t]:=k(t,\cdot) \quad\mbox{and}\quad \KK^{-1}[\1_t]:=k^{-1}(t,\cdot).
$$
Then, just like with the operators $\PP$ and $\PP^{-1}$, we have
\begin{eqnarray*}
\int_0^T f(t)\, \d V_t &=& \int_0^T \KK[f](t)\, \d W_t, \\
\int_0^T g(t)\, \d W_t &=& \int_0^T \KK^{-1}[g](t)\, \d V_t.
\end{eqnarray*}
The connection between the operators $\KK$ and $\KK^{-1}$ and the operators
$\PP$ and $\PP^{-1}$ are
\begin{eqnarray*}
\KK[g] &=& k(T,\cdot)\PP^{-1}[g], \\
\KK^{-1}[g] &=& k^{-1}(T,\cdot)\PP[g].
\end{eqnarray*}
So, invertible Gaussian Volterra processes are prediction-invertible and
the following corollary to Theorem \ref{thm:canonical-pred-inv} is obvious:

\begin{cor}\label{cor:canonical-prev-inv}
Let $V$ be an invertible Gaussian Volterra process and let 
$\KK[g_i]\in L^2([0,T])$ for all $i=1,\ldots,N$. Denote
$$
\tilde \g(t) := \KK[\g](t).
$$
Then the bridge $V^\g$ admits the canonical representation
\begin{equation}\label{eq:bridge-gvp} 
V^\g_t = V_t - 
\int_0^t \int_s^t k(t,u)
\KK^{-1}\left[\tilde\ell_{\tilde\g}^*(u,\cdot)
\right](s)\, \d u \, \d V_s,
\end{equation}
where
\begin{eqnarray*}
\tilde \ell_{\tilde\g}(u,v) &=&
|{\lla \tilde\g \rra}^{W}|(u)\tilde\g^\top(u)
{(\lla \tilde\g \rra^{W})}^{-1}(u)\,
\frac{\tilde\g(v)}{|\lla \tilde\g \rra^{W}|(v)}, \\
\lla \tilde\g \rra^{W}_{ij}(t) &=&
\int_t^T \tilde g_i(s)\tilde g_j(s)\, \d s 
= \lla \g \rra^{V}_{ij}(t).
\end{eqnarray*}
\end{cor}

\begin{exa}\label{exa:fbm}
The \emph{fractional Brownian motion} $B=(B_t)_{t\in[0,T]}$ with Hurst index
$H\in(0,1)$ is a centered Gaussian process with $B_0=0$ and covariance function
$$
R(t,s) = \frac12\left(t^{2H}+s^{2H}-|t-s|^{2H}\right).
$$
Another way of defining the fractional Brownian motion is that it is the 
unique centered Gaussian $H$-self-similar process with stationary increments
normalized so that $\E[B_1^2]=1$.

It is well-known that the fractional Brownian motion is an invertible Gaussian 
Volterra process with
\begin{eqnarray} \label{eq:kk-fbm}
\KK[f](s) &=& c_H 
s^{\frac12-H} \II_{T-}^{H-\frac12}\left[(\,\cdot\,)^{H-\frac12}f\right](s), \\
\KK^{-1}[f](s) &=& \frac{1}{c_H} \label{eq:kk-inverse-fbm}
s^{\frac12-H} \II_{T-}^{\frac12-H}\left[(\,\cdot\,)^{H-\frac12}f\right](s).
\end{eqnarray}
Here $\II_{T-}^{H-\frac12}$ and $\II_{T-}^{\frac12-H}$ are the 
\emph{Riemann-Liouville fractional integrals} over $[0,T]$ of order
$H-\frac12$ and $\frac12-H$, respectively: 
$$
\II_{T-}^{H-\frac12}[f](t) = \left\{\begin{array}{rr}
\frac{1}{\Gamma(H-\frac12)}\int_t^T \frac{f(s)}{(s-t)^{\frac32-H}}\, \d s,
& \mbox{for } H>\frac12,\\
\frac{-1}{\Gamma(\frac32-H)}\frac{\d}{\d t}
\int_t^T \frac{f(s)}{(s-t)^{H-\frac12}}\, \d s, 
&\mbox{for } H<\frac12,
\end{array}\right.
$$
and $c_H$ is the normalizing constant 
$$
c_H=\left( \frac{2H\Gamma(H+\frac12)\Gamma(\frac32-H)}{
\Gamma(2-2H)}\right)^{\frac12}.
$$
Here
$$
\Gamma(x) = \int_0^\infty e^{-t}t^{x-1}\, \d t
$$ 
is the Gamma function.
For the proofs of these facts and for more information on the fractional
Brownian motion we refer to the monographs  by Biagini et al.
\cite{BiaginiHuOksendalZhang} and Mishura \cite{Mishura}.

One can calculate the canonical representation for generalized fractional
Brownian bridges by using the representation \eqref{eq:bridge-gvp} by plugging
in the operators $\KK$ and $\KK^{-1}$ defined by \eqref{eq:kk-fbm} and 
\eqref{eq:kk-inverse-fbm}, respectively.  Unfortunately, even for 
a simple bridge the formula becomes very complicated.  Indeed, consider the 
standard fractional Brownian bridge $B^1$, i.e., the conditioning is 
$g(t) = \1_T(t)$.
Then
$$
\tilde g(t) = \KK[\1_T](t) = k(T,t)
$$
is given by \eqref{eq:kk-fbm}. Consequently,
\begin{eqnarray*}
\lla \tilde g \rra^W(t) &=& 
\int_t^T k(T,s)^2 \, \d s, \\
\tilde\ell_{\tilde g}^*(u,v) &=& 
k(T,u)\frac{k(T,v)}{\int_v^T k(T,w)^2\, \d w}.
\end{eqnarray*}
We obtain the canonical representation for the fractional Brownian bridge:
\begin{eqnarray*}
B^1_t &=& B_t - \int_0^t \int_s^t k(t,u)k(T,u)
\KK^{-1}\left[\frac{k(T,\cdot)}{\int_{\cdot}^T k(T,w)^2\, \d w}\right](s)\, 
\d u\, \d B_s.
\end{eqnarray*}
This representation can be made ``explicit'' by plugging in the definitions 
\eqref{eq:kk-fbm} and \eqref{eq:kk-inverse-fbm}.  It seems, however, very difficult
to simplify the resulting formula.
\end{exa}

\section{Application to Insider Trading}

We consider insider trading in the context of initial enlargement of 
filtrations.  Our approach here is motivated by Amendinger \cite{Amendinger}
and Imkeller \cite{Imkeller}, where only one condition were used.  We extend that 
investigation to multiple conditions although otherwise our investigation is less 
general than in \cite{Amendinger}.

Consider an insider who has at time $t=0$ some insider information of the 
evolution of the price process of a financial asset $S$ over a period $[0,T]$. 
We want to calculate the additional expected utility for the insider trader. 
To make the maximization of the utility of terminal wealth reasonable
we have to assume that our model is arbitrage-free. In our Gaussian
realm this boils down to assuming that the (discounted) asset prices
are governed by the equation  
\begin{equation}\label{eq:model}
\frac{\d S_t}{S_t}  =  a_t\,\d \la M\ra_t + \d M_t,
\end{equation}
where $S_0=1$, $M$ is a continuous Gaussian martingale with strictly 
increasing $\la M\ra$ with $M_0=0$, and the process $a$ is
$\mathbb{F}$-adapted satisfying 
$\int_0^T a_t^2 \, \d \langle M\rangle_t < \infty \; \P$-a.s.

Assuming that the trading ends at time $T-\epsilon$, the insider knows some 
functionals of the return over the interval $[0,T]$. If $\epsilon=0$ there is 
obviously arbitrage for the insider.  The insider information will define a 
collection of  functionals $G_T^i(M) = \int_0^T g_i(t)\, \d M_t$, 
where $g_i \in L^2([0,T],\d\langle M\rangle),\, i=1, \ldots, N$, such that 
\begin{equation}\label{eq:info}
\int_0^T \g(t) \frac{\d S_t}{S_t} = \y = [y_i]_{i=1}^N,
\end{equation}
for some $\y\in \mathbb{R}^N$. This is equivalent to the conditioning of the 
Gaussian martingale $M$ on the  linear functionals $\G_T=[G_T^i]_{i=1}^N$ of 
the log-returns:
\begin{equation*}
\G_T(M) = \int_0^T \g(t)\, \d M_t
= \left[\int_0^T g_i(t)\, \d M_t\right]_{i=1}^N.
\end{equation*} 
Indeed, the connection is
$$
\int_0^T \g(t) \, \d M_t = \y-\lla a,\g\rra=: \y',
$$
where
$$\lla a, \g\rra = [\lla a, g_i\rra]_{i=1}^N
= \left[\int_0^T a_t g_i(t)\, \d \la M\ra_t \right]_{i=1}^N.
$$ 
As the natural filtration $\mathbb{F}$ represents the information available to 
the ordinary trader, the insider trader's information flow is described by a 
larger filtration $\mathbb{G}=(\mathcal{G}_t)_{t\in [0,T]}$ given by
$$
\mathcal{G}_t=\F_t\vee \sigma(G_T^1, \dots, G_T^N).
$$
Under the augmented filtration $\mathbb{G}$, $M$ is no longer a martingale. It 
is a Gaussian semimartingale  with the \emph{semimartingale decomposition}
\begin{equation}\label{mg semimg}
\d M_t = 
\d \Tilde{M}_t+\left(\int_0^t \ell_\g(t,s)\, \d M_s-
\g^\top(t){\lla\g\rra}^{-1}(t)\y'\right)\, \d \la M\ra_t,
\end{equation}
where $\tilde{M}$ is a continuous $\mathbb{G}$-martingale with bracket 
$\langle M\rangle$, and which can be constructed through the formula 
(\ref{eq:can brg 2}).
 
In this market, we consider the portfolio process $\pi$ defined on 
$[0,T-\epsilon]\times\Omega$  as the fraction  of the total wealth invested in 
the asset $S$. So the dynamics of the discounted value process associated to  
a self-financing strategy $\pi$  is defined by $V_0=v_0$ and  
$$
\frac{\d V_t}{V_t}=\pi_t\frac{\d S_t}{S_t} , \quad \text{for } \, 
t\in [0,T-\epsilon],
$$
or equivalently by
\begin{equation}\label{eq:wlth}
V_t=v_0\exp\left(\int_0^t \pi_s \, \d M_s+\int_0^t\left(\pi_s a_s
-\frac12 \pi_s^2\right) \, \d \la M\ra_s\right).
\end{equation}

Let us denote by $\lla \cdot, \cdot \rra_\epsilon$ and 
$\vvbar \cdot \vvbar_\epsilon$ the inner product and the norm on 
$L^2([0,T-\epsilon],\d\langle M\rangle)$. 

For the ordinary trader, the process $\pi$ is assumed to be a non-negative
$\mathbb{F}$-progressively measurable process such that 
\begin{enumerate}
\item $\P[\vvbar \pi \vvbar_\epsilon^2 < \infty]=1$.
\item $\P[\lla \pi,f \rra_\epsilon< \infty]=1$,  for all 
$f \in L^2([0,T-\epsilon],\d\langle M\rangle)$.
\end{enumerate}
 We denote this class of  portfolios by $\Pi(\mathbb{F})$. By analogy, the 
class of the portfolios disposable to the insider trader shall be denoted by 
$\Pi(\mathbb{G})$, the class of  non-negative $\mathbb{G}$-progressively 
measurable processes that satisfy the conditions (i) and (ii) above.

The aim of both investors is to maximize the expected utility of the terminal 
wealth $V_{T-\epsilon}$, by finding an optimal portfolio $\pi$ on 
$[0,T-\epsilon]$ that solves the optimization problem
$$
\max_{\pi}\E\left[U(V_{T-\epsilon})\right].
$$
Here, the utility function $U$ will be the logarithmic utility function, and 
the utility of the process  (\ref{eq:wlth}) valued at time $T-\epsilon$ is  
\begin{eqnarray} \label{log wealth}
\lefteqn{\log V_{T-\epsilon} } \\ \nonumber
&=&
\log v_0 + \int_0^{T-\epsilon} \pi_s \, \d M_s+\int_0^{T-\epsilon}
\left(\pi_s a_s
-\frac12 \pi_s^2\right) \, \d \la M\ra_s \\ \nonumber
&=&
\log v_0 + \int_0^{T-\epsilon} \pi_s \, \d M_s+ \frac12\int_0^{T-\epsilon}
\pi_s\left(2a_s -\pi_s\right) \, \d \la M\ra_s \\ \nonumber
&=&
\log v_0+ \int_0^{T-\epsilon} \pi_s \, \d M_s
+\frac12\lla \pi , 2a-\pi \rra_{\epsilon} 
\end{eqnarray}

From the ordinary trader's point of view $M$ is a martingale. So,
$\E\left(\int_0^{T-\epsilon} \pi_s \, \d M_s\right)=0$ for every
$\pi \in \Pi(\mathbb{F})$ and, consequently,
$$
\E\left[U(V_{T-\epsilon})\right]=\log v_0+
\frac12 \E\left[\lla \pi, 2a-\pi\rra_{\epsilon}\right].
$$ 
Therefore, the ordinary trader, given $\Pi(\mathbb{F})$, will solve the 
optimization problem 
$$
\max_{\pi\in \Pi(\mathbb{F})}\E\left[U(V_{T-\epsilon})\right]=
\log v_0+
\frac12 \max_{\pi\in \Pi(\F)}\E\left[\lla \pi, 2a-\pi\rra_{\epsilon}\right]
$$
over the term $\lla \pi, 2a-\pi\rra_{\epsilon}=2\lla \pi, a\rra_{\epsilon}-
\vvbar \pi \vvbar_{\epsilon}^2$. By using the polarization identity we obtain
$$
\lla \pi, 2a-\pi\rra_{\epsilon}=\vvbar  a \vvbar_{\epsilon}^2-\vvbar \pi
- a \vvbar_{\epsilon}^2 \leq \vvbar  a \vvbar_{\epsilon}^2.
$$
Thus, the maximum is obtained with the choice $\pi_t=a_t$ for $t\in [0,T-\epsilon]$, and maximal expected utility value is
$$
\max_{\pi\in \Pi(\mathbb{F})}\E\left[U(V_{T-\epsilon})\right]=
\log v_0+\frac12\E\left[\vvbar a \vvbar_{\epsilon}^2\right] .
$$

From the insider trader's point of view the  process $M$ is not a martingale 
under his information flow $\mathbb{G}$. The insider can update his 
utility of terminal wealth (\ref{log wealth}) by considering 
(\ref{mg semimg}), where $\tilde{M}$ is a continuous $\mathbb{G}$-martingale. 
This gives
\begin{eqnarray*}
\log V_{T-\epsilon}&=&\log v_0+\int_0^{T-\epsilon} \pi_s \, \d \tilde{M}_s
+\frac12\lla \pi, 2a-\pi\rra_{\epsilon} \\
& &+ \lla \pi, \int_0^{\cdot} \ell_\g(\cdot,t) \, \d M_t - 
\g^\top{\lla \g \rra}^{-1}\y' \rra_{\epsilon}.
\end{eqnarray*}
Now, the insider maximizes the expected utility over all 
$\pi\in \Pi(\mathbb{G})$:
\begin{eqnarray*}
\lefteqn{\max_{\pi\in \Pi(\mathbb{G})}\E\left[ \log V_{T-\epsilon}\right] }\\
&=&\log v_0\\
& &
+\frac12 \max_{\pi\in \Pi(\mathbb{G})}\E \left[ \lla \pi, 2\left(a+ 
\int_0^{\cdot} \ell_\g(\cdot,t)\, \d M_t - \g^\top{\lla\g \rra}^{-1}\y'\right)-
\pi \rra_{\epsilon}\right].
\end{eqnarray*}
The optimal portfolio $\pi$ for the insider trader can be computed in the same
way as for the ordinary trader. We obtain the optimal portfolio
$$
\pi_t=a_t+ \int_0^t \ell_\g(t,s) \, \d M_s - \g^\top(t)
{\lla\g \rra}^{-1}(t)\y', \quad t \in [0,T-\epsilon].
$$

Let us then calculate the additional expected logarithmic utility for the 
insider trader. Since 
$$
\E\left[\lla a, \int_0^{\cdot} \ell_\g(\cdot,s) \, \d M_s- 
\g^\top{\lla \g \rra}^{-1}\y' \rra_{\epsilon} \right]= 0, 
$$
we obtain that 
\begin{eqnarray*}
\Delta_{T-\epsilon} &=& 
\max_{\pi\in \Pi(\mathbb{G})}\E\left[U(V_{T-\epsilon})\right]
-\max_{\pi\in \Pi(\mathbb{F})}\E\left[U(V_{T-\epsilon})\right]\\
&=&
\frac12\E\left[\vvbar  \int_0^{\cdot} \ell_\g(\cdot,s) \, \d M_s-
\g^\top{\lla\g \rra}^{-1}\y' \vvbar_{\epsilon}^2 \right]. 
\end{eqnarray*}
Now, let us use the short-hand notation
\begin{eqnarray*}
\G_t &:=& \int_0^t \g(s)\, \d M_s, \\
\lla\g\rra(t,s) &:=& \lla\g\rra(t)-\lla\g\rra(s), \\
\lla\g\rra^{-1}(t,s) &:=& \lla\g\rra^{-1}(t)-\lla\g\rra^{-1}(s).
\end{eqnarray*}
Then, by expanding the square $\vvbar\cdot\vvbar_\epsilon^2$,
we obtain
\begin{eqnarray*}
2\Delta_{T-\epsilon} &=& 
\E\left[\vvbar  \int_0^{\cdot} \ell_\g(\cdot,s) \, \d M_s
-\g^\top{\lla\g \rra}^{-1}\y' \vvbar_{\epsilon}^2 \right] \\
&=&
\E\left[\vvbar\g^\top{\lla\g\rra}^{-1}\left(\y'+
\G\right)  \vvbar_{\epsilon}^2 \right]\\
&=& 
\E\left[\int_0^{T-\epsilon} \y'^\top {\lla\g\rra}^{-1}(t) 
\g(t)\g^\top(t){\lla\g\rra}^{-1}(t)\y' \, \d \langle M\rangle_t \right] \\
& & 
+\E\left[\int_0^{T-\epsilon} \G_t^\top 
{\lla\g\rra}^{-1}(t)\g(t)\g^\top(t){\lla\g\rra}^{-1}(t)
\G_t \, \d \langle M\rangle_t \right].
\end{eqnarray*}
Now the formula 
$\E[\x^\top\mathbf{A}\x] = \Tr[\mathbf{A}\Cov\x] 
+ \E[\x]^\top\mathbf{A}\E[\x]$ yields
\begin{eqnarray*}
2\Delta_{T-\epsilon}
&=& 
\int_0^{T-\epsilon} \y'^\top {\lla\g\rra}^{-1}(t)\g(t)\g^\top(t)
{\lla\g\rra}^{-1}(t)\y' \, \d \langle M\rangle_t \\
& &
+\int_0^{T-\epsilon} \Tr 
\left[{\lla\g\rra}^{-1}(t)\g(t)\g^\top(t){\lla\g\rra}^{-1}(t)
{\lla\g\rra}(0,t)\right] 
\, \d \langle M\rangle_t\\
&=& 
\y'^\top{\lla\g\rra}^{-1}(T-\epsilon,0)\y'\\
& &+
\int_0^{T-\epsilon} \Tr \left[{\lla\g\rra}^{-1}(t) \g(t)\g^\top(t)
{\lla\g\rra}^{-1}(t){\lla\g\rra}(0)
\right] \, \d \langle M\rangle_t\\
& &-
\int_0^{T-\epsilon} \Tr \left[{\lla\g\rra}^{-1}(t) \g(t)\g^\top(t) 
\right] \, \d \langle M\rangle_t\\
&=&
\left(\y-\lla \g, a\rra\right)^\top 
{\lla\g\rra}^{-1}(T-\epsilon,0)
\left(\y-\lla \g, a\rra\right)\\
& &+
\Tr\left[{\lla\g \rra}^{-1}(T-\epsilon,0)\lla\g\rra(0)\right]\\
& &+
\log \frac{|\lla\g\rra|(T-\epsilon)}{|\lla\g \rra|(0)}.
\end{eqnarray*}
We have proved the following proposition:

\begin{pro}\label{pro:insider}
The additional logarithmic utility in the model (\ref{eq:model}) for
the insider with information (\ref{eq:info}) is
\begin{eqnarray*}
\Delta_{T-\epsilon} &=& 
\max_{\pi\in \Pi(\mathbb{G})}\E\left[U(V_{T-\epsilon})\right]-
\max_{\pi\in \Pi(\mathbb{F})}\E\left[U(V_{T-\epsilon})\right]\\
&=&\frac12 \left(\y-\lla\g, a\rra\right)^\top \left(
{\lla\g \rra}^{-1}(T-\epsilon)-{\lla\g \rra}^{-1}(0)\right)
\left(\y-\lla\g, a\rra\right)\\
& &+\frac12 \Tr\left[\left({\lla\g\rra}^{-1}(T-\epsilon)-{\lla\g\rra}^{-1}(0)
\right)\lla \g \rra(0)\right]\\
& &+\frac12 \log \frac{|\lla \g \rra|(T-\epsilon)}{|\lla \g \rra|(0)}.
\end{eqnarray*}
\end{pro}

\begin{exa}
Consider the classical Black and Scholes pricing model:
$$
\frac{\d S_t}{S_t}=\mu \d t + \sigma \d W_t, \qquad S_0=1,
$$
where $W=\left(W_t\right)_{t\in[0,T]}$ is the standard Brownian motion.
Assume that the insider trader knows at time $t=0$ that the total and the 
average return of the stock price over the period $[0,T]$ are both zeros
and that the trading ends at time $T-\epsilon$. 
So, the insider knows that
\begin{eqnarray*}
G_T^1 &=& \int_0^T g_1(t)\, \d W_t=\frac{y_1}{\sigma}-\frac{\mu}{\sigma}
\lla g_1,\1_T \rra=-\frac{\mu}{\sigma}\lla g_1,\1_T \rra \\
G_T^2 &=& \int_0^T g_2(t)\, \d W_t=\frac{y_2}{\sigma}-\frac{\mu}{\sigma}
\lla g_2,\1_T \rra=-\frac{\mu}{\sigma}\lla g_2,\1_T \rra,
\end{eqnarray*}
where
\begin{eqnarray*}
g_1(t) &=& \1_T(t), \\
g_2(t) &=& \frac{T-t}{T}. 
\end{eqnarray*}
Then, by Proposition \ref{pro:insider},
\begin{eqnarray*}
\Delta_{T-\epsilon} &=&\frac12\left(\frac{\mu}{\sigma}\right)^2
\lla\g,\1_T \rra^\top 
\left({\lla\g \rra}^{-1}(T-\epsilon)-{\lla\g\rra}^{-1}(0)\right)
\lla\g,\1_T \rra\\
& &
+\frac12 \Tr\left[
\left({\lla\g\rra}^{-1}(T-\epsilon)-{\lla\g\rra}^{-1}(0)\right)\lla\g\rra(0)
\right]\\
& &+\frac12 \log \frac{|\lla\g \rra|(T-\epsilon)}{|\lla\g \rra|(0)},
\end{eqnarray*}
with 
$$
{\lla\g\rra}^{-1}(t)
=
\left[\begin{array}{rr}
\frac{4}{T}\left(\frac{T}{T-t}\right) &  
-\frac{6}{T}\left(\frac{T}{T-t}\right)^2\\
-\frac{6}{T}\left(\frac{T}{T-t}\right)^2 & 
\frac{12}{T}\left(\frac{T}{T-t}\right)^3
\end{array}\right]
$$
for all $t \in [0,T-\epsilon]$. We obtain
\begin{eqnarray*}
\Delta_{T-\epsilon}&=& 
\frac12\left(\frac{\mu}{\sigma}\right)^2\left\{3T\left(\frac{T}{\epsilon}
\right)^3-6T\left(\frac{T}{\epsilon}\right)^2+4T\left(\frac{T}{\epsilon}
\right)-T\right\}\\
& &+2\left(\frac{T}{\epsilon}\right)^3-3\left(\frac{T}{\epsilon}\right)^2 +2
\left(\frac{T}{\epsilon}\right)-2\log\left(\frac{T}{\epsilon}\right)-1.
\end{eqnarray*}
Here it can be nicely seen that $\Delta_0=0$ (no trading at all) and 
$\Delta_T=\infty$ (the knowledge of the final values implies arbitrage). 
\end{exa}


\end{document}